\documentclass[12pt]{amsart}
\usepackage{amssymb,amsmath,enumitem}
\usepackage[foot]{amsaddr}
\usepackage{graphicx, xcolor}
\usepackage{comment}
\usepackage{stackengine}

\newcommand{\R}{\mathbb{R}}

\newcommand{\N}{\mathbb{N}}

\newcommand{\W}{{\bf{W}}}
\newcommand{\I}{{\bf{I}}}

\newcommand{\be}{\begin{equation}}
\newcommand{\ee}{\end{equation}}
\newcommand{\bee}{\begin{equation*}}
\newcommand{\eee}{\end{equation*}}
\newcommand{\bea}{\begin{eqnarray}}
\newcommand{\eea}{\end{eqnarray}}
\newcommand{\bess}{\begin{eqnarray*}}
\newcommand{\eess}{\end{eqnarray*}}

\numberwithin{equation}{section}
\usepackage{hyperref}
\hypersetup{colorlinks=true, citecolor=teal, linkcolor=magenta, urlcolor=blue}
\theoremstyle{plain}
\newtheorem{Thm}{Theorem}[section]
\newtheorem{Cor}[Thm]{Corollary}

\theoremstyle{definition}
\newtheorem{Def}[Thm]{Definition}
\newtheorem{Rem}[Thm]{Remark}

\theoremstyle{remark}

\begin{document}
\title[Fractional Sublinear Sobolev inequality]
{Fractional sublinear Sobolev inequality for $\mathcal{L}-$superharmonic functions}
  
\author{Aye Chan May}
\address[Aye Chan May]{Sirindhorn International Institute of Technology, Thammasat University, Pathum Thani 12120, Thailand}
\email{\href{d6622300199@g.siit.tu.ac.th}{d6622300199@g.siit.tu.ac.th (A.C. May)}}

\author{Adisak Seesanea}
\address[Corresponding Author: Adisak Seesanea]{Sirindhorn International Institute of Technology, Thammasat University, Pathum Thani 12120, Thailand}
\email[A. Seesanea]{\href{adisak.see@siit.tu.ac.th}{adisak.see@siit.tu.ac.th (A. Seesanea)}}

\maketitle
\begin{abstract}
We establish a Sobolev-type inequality in Lorentz spaces for $\mathcal{L}$-superharmonic functions
\[
\|u\|_{L^{\frac{nq}{n-\alpha q},t}(\mathbb{R}^n)} \leq c \left\| \frac{u(x) - u(y)}{|x-y|^{\frac{n}{q}+\alpha}} \right\|_{L^{q,t}(\mathbb{R}^n \times \mathbb{R}^n)}
\]
in the sublinear case $p-1 < q < 1$ and $p-1\leq t\leq \infty$. The nonlocal nonlinear elliptic operator 
$\mathcal{L}$ is modeled from the fractional $p$-Laplacian $(- \Delta_{p})^{\alpha} $ with $0 < \alpha < 1$
and $1<p<2$. Related Gagliardo-Nirenberg interpolation for $\mathcal{L}$-superharmonic functions is also derived.
\end{abstract}
\setcounter{tocdepth}{1}
\section{Introduction}\label{sect:intro} 

The classical Sobolev inequality by Sobolev \cite{So} states that for any $1<q<n$ there exists a positive constant $C(n,q)$ such that 
\be\label{standardsobolev}
\big(\int_{\mathbb{R}^{n}} u^{\frac{nq}{n-q}} dx \big)^{\frac{n-q}{nq}}\leq C(n,q)\big(\int_{\mathbb{R}^{n}} |\nabla u|^{q} dx\big)^{\frac{1}{q}}
\ee
for every function $u\in\mathcal{C}^{\infty}_{0}(\R^{n})$.

A natural extension of \eqref{standardsobolev} is in the framework of fractional Sobolev space $W^{\alpha,q}(\R^{n})$. In this setting, the following fractional Sobolev inequality holds, for any $1<q<\frac{n}{\alpha}$ and $0<\alpha<1$,
\be\label{original}
\big( \int_{\R^{n}} u^{\frac{nq}{n-\alpha q}} dx\big)^{\frac{n-\alpha q}{nq}}\leq C(n,\alpha,q)\Big( \int_{\R^{n}}\int_{\R^{n}}\frac{|u(x)-u(y)|^{q}}{|x-y|^{n+\alpha q}} dydx\Big)^{\frac{1}{q}}.
\ee
for all $u\in W^{\alpha,q}(\R^{n})$.  There are a large number of references in the literature studying the above inequality with different approaches, see \cite{EGE}, \cite{Po}, and \cite{Sa}.

 The general idea described here had been motivated by \cite{Ph1}. In particular, the author showed that \eqref{standardsobolev} remains to hold for $p-1<q<1$ (sublinear case) when restricted to a class of positive $p$-superharmonic ($1<p<2$) functions $u$ in $\R^{n}$ such that $\inf\limits_{\R^{n}} u=0.$ The techniques in \cite{Ph1} are based on the global pointwise estimates in terms of Wolff potential, as established by Kilpel\"{a}inen and Mal\'{y} \cite{KM1}. Furthermore, the authors in \cite{PV23} proved that the uniqueness of $\mathcal{A}$-superharmonic solution  to the equation 
 \[
 -\text{div} \mathcal{A}(x,\nabla u)=\sigma\quad\text{in}\quad \R^{n},
 \]
under the conditions involving the weak integrability of the gradient of the solution, utilizing the sublinear Sobolev inequality (see \cite[Theorem 3.12]{PV23}.  In addition, Chua and Phuc extended the sublinear Sobolev inequality for any superharmonic functions on the generalized John domain, see \cite[Corollary 5.7]{CP}.

 Although the sublinear Sobolev inequality in \cite{Ph1} holds for a more general class of positive $p$-superharmonic functions, the author in \cite{Da} pointed out that the classical Sobolev inequality fails for all functions $u\in C_{0}^{\infty}(\mathbb{R}^{n})$ with $0<q<1$ by providing a counterexample.

This paper aims to establish that the sublinear fractional Sobolev inequality of the type \eqref{original} under Lorentz norm holds when the Sobolev exponent $p-1<q<1$ for a class of $\mathcal{L}$-superharmonic functions $u$ on $\Omega$ with $u=0$ in $\R^{n}\setminus\Omega$. We demonstrate that a similar sublinear fractional Sobolev inequality remains valid when  Lorentz norm on both sides are replaced by $L^{p}$ norm, providing a fractional analogue of the result in the \cite{Ph1}.  More generally, we also consider fractional Sobolev inequality with Muckenhoupt weight in the sublinear case (see Theorem \ref{weighted}).

In this paper, we consider a nonlocal operator with nonlinear growth defined on an open and bounded subset $\Omega\subset\R^{n}$ with $n\in\N$. For $0<\alpha<1,p>1 $ and $\Lambda\geq 1$,
\[
\mathcal{L} u(x)=2 \lim\limits_{\varepsilon\to 0}\int_{\R^n\setminus B(x,\varepsilon)} |u(x)-u(y)|^{p-2}(u(x)-u(y))k(x,y) dy,\;\; x\in\R^{n}
\]
where the kernel $k:\R^{n}\times\R^{n}\to [0,\infty),$ which is a symmetric measurable function satisfying the following condition:
\[
\Lambda^{-1} |x-y|^{-n-\alpha p}\leq k(x,y)\leq \Lambda |x-y|^{-n-\alpha p}.
\]

In the case $k(x,y)=|x-y|^{-n-\alpha p},$ this operator is modeled on the fractional $p$-Laplacian, $(-\Delta_{p})^{\alpha}.$ Notably, in the linear nonlocal case when $p=2$ and $\Lambda=1,$ this reduces to the pure fractional Laplacian operator $(-\Delta)^{\alpha}$. We refer to \cite{Pa} for the introduction of this operator.

\begin{Thm}\label{mainthm}
Let $\Omega$ be an open bounded subset in $\R^{n}$ where $n\geq 1$. Let $0<\alpha<1, p-1\leq t\leq \infty, 1<p<2$ and $p-1<q<1.$ For any $\mathcal{L}$-superharmonic function $u$ on $\Omega$ with $u=0$ in $\R^{n}\setminus\Omega$, we have
\be \label{fracSoboineq}
\|u\|_{L^{\frac{nq}{n-\alpha q},t}(\R^{n})} \leq C [u]_{W^{\alpha,(q,t)}(\R^{n})}
\ee
where $C=C(n,\alpha,p,q,t,\Lambda)>0.$
\end{Thm}
The proof of Theorem \ref{mainthm} is given in Section \ref{sec3}.
\subsection*{Organization of the paper} This paper is organized into three sections. In section \ref{sec2}, we give some background information on $\mathcal{L}$ -superharmonic, a review of the theory of the Triebel-Lizorkin-Lorentz spaces, and some useful theorems in our work. The section \ref{sec3} is devoted to the proof of the main theorem. We discuss further results related to the fractional Sobolev inequality in the final section.
\section{Preliminaries}\label{sec2}
\subsection*{Function spaces}
To state our results, let us introduce some definitions and notation. Let $\Omega$ be an bounded open set in $\R^{n}$ with $n\in\N$ and $\mathcal{M}^{+}(\Omega)$ denote the class of all nonnegative finite Borel measures on $\Omega$. We denote $\rho(\R^{n})$ and  $\rho'(\R^{n})$ as the Schwartz space of all complex-valued rapidly decreasing infinitely differentiable functions on $\R^{n}$ and the set of all tempered distributions, respectively. 

Let $0<p<\infty,$ $1\leq r\leq\infty$ and $\Omega$ be an open set in $\R^{n}$. Let $L^{p,r}$ denote the \textbf{Lorentz space} of a measurable function $u,$ which is a quasi-Banach space equipped with the quasi-norm
\[
\|u\|_{L^{p,r}(\Omega)}=
\begin{cases}
\Big( p\int_{0}^{\infty}[\beta \mu_{u}(\beta)^{\frac{1}{p}}]^{r}\frac{d\beta}{\beta}\Big)^{\frac{1}{r}}\quad\text{if}\quad p,r<\infty,\\
\sup\limits_{\beta>0}\beta \mu_{u}(\beta)^{\frac{1}{p}}\quad\quad\quad\quad\quad\;\text{if}\quad p<\infty, q=\infty,\\
\inf\{\beta>0:\mu_{u}(\beta)=0\} \;\;\;\text{if}\quad p=r=\infty,
\end{cases}
\]
where $\mu_{u}$ denotes the distribution function of $u:\mu_{u}(\beta)=|\{x\in\Omega:|u(x)|>\beta\}|.$

For $0<\alpha<1,$ $0<p<\infty,$ and $1\leq r\leq\infty$,  we define the \textbf{fractional Sobolev-Lorentz space} $W^{\alpha,(p,r)}$ as follows
\[
W^{\alpha,(p,r)}(\Omega)=\{u\in L^{p,r}(\Omega): \frac{|u(x)-u(y)|}{|x-y|^{\frac{n}{p}+\alpha}} \in L^{p,r}(\Omega\times\Omega)\}.
\]
 It is a quasi-Banach space and associated with this norm is called the \textbf{Gagliardo seminorm} 
 \[
 [u]_{W^{\alpha,(p,r)}(\Omega)}=\Big\| \frac{|u(x)-u(y)|}{|x-y|^{\frac{n}{p}+\alpha }} \Big\|_{L^{p,r}(\Omega\times\Omega)}.
 \]
 We refer to readers \cite{EE} and \cite{Le} for more basic properties of fractional Sobolev space.

Next, we recall the standard definition of Triebel-Lizorkin-Lorentz space. To define the space, we choose a sequence of test functions $\{\phi_{j}\}_{j=0}^{\infty}\subset \rho$ such that
 \begin{itemize}
 \item[(i)] there exist positive constants $A,B,C$ and
 \[
 \begin{cases}
\mathrm{supp}\; \phi_{0}\subset\{x| \; |x|\leq A\},\\
\mathrm{supp}\; \phi_{j}\subset\{x|  B 2^{j-1}\leq |x| \leq C 2^{j+1}\}\quad\text{if}\quad j=1,2,3,\dots,
 \end{cases}
 \]
 \item[(ii)] for every multi-index $\alpha$ there exists a positive cumber $c_{\alpha}$ and
 \[
 \sup_{x}\sup_{j=0,1,\dots} 2^{j|\alpha|}|D^{\alpha}\phi_{j}(x)|\leq c_{\alpha},
 \]
 \item[(iii)] \[
 \sum_{j=0}^{\infty}\phi_{j}(x)=1\quad\text{for every}\quad x\in\R^{n}.
 \]
 \end{itemize}
 
 \begin{Def}\label{Triebel}
 Let $s$ be real, $0<q,r\leq \infty$ and $0<p<\infty$. The \textbf{Triebel-Lizorkin-Lorentz  space} ${F}^{\alpha}_{q}[L^{p,r}]$ is the space of all tempered distributions $f$ such that
 \[
 \Big\| \Big(\sum_{j=0}^{\infty} 2^{j\alpha q} |\mathcal{F}^{-1}(\phi_{j}(\xi)\hat{f}(\xi))|^{q}\Big)^{\frac{1}{q}} \Big\|_{L^{p,r}(\R^{n})} <\infty.
 \]
 \end{Def}
 Clearly, in the case $p=r$, Definition \ref{Triebel} covers the usual Triebel-Lizorkin spaces ${F}^{\alpha}_{q}[L^{p,p}](\R^{n})=F^{\alpha}_{p}(\R^{n}).$
\begin{Rem}
The equivalent quasi-norms of the Triebel-Lizorkin norm are as follows:
\begin{itemize}
\item[(i)] $\|\cdot\|_{L^{p}(\R^{n})}=\|\cdot\|_{F^{0}_{p,2}(\R^{n})}$\;\;\quad\quad if \quad $1<p<\infty,$
\item[(ii)] $\|\cdot\|_{W^{\alpha,p}(\R^{n})}=\|\cdot\|_{F^{\alpha}_{p,p}(\R^{n})}$\;\;\quad if \quad $1\leq p<\infty, 0<\alpha\neq$integer,
\item[(iii)] $\|\cdot\|_{L^{p,r}(\R^{n})}=\|\cdot\|_{F^{0}_{2}[L^{p,r}](\R^{n})}$ \;if\quad\;$1<p<\infty$ and $0<r\leq \infty.$
\end{itemize}
These equivalences in the Triebel-Lizorkin and Triebel-Lizorkin Lorentz spaces have been studied by Triebel \cite{Tr}, Runst-Sickel \cite{RS}, as well as Yang-Cheng-Bang \cite{YCB}.
\end{Rem}
\begin{Thm}[{\text See \cite[Theorem 1.6]{ST} }]  \label{Lorentzembedding}
Let $\alpha_{1},\alpha_{2}\in\R, 0<p_{1},p_{2}<\infty, 0<q_{1},q_{2},r_{1},r_{2}\leq\infty.$ The embedding
\[
\|\cdot\|_{F^{\alpha_{2}}_{q_{2}}[L^{{p_{2},r_{2}}}](\R^{n})}\leq \|\cdot\|_{F^{\alpha_{1}}_{q_{1}}[L^{{p_{1},r_{1}}}](\R^{n})}
\]
if and only if the condition is satisfied
\[
\quad \alpha_{1}-\frac{n}{p_{1}}=\alpha_{2}-\frac{n}{p_{2}}, r_{1}\leq r_{2}.
\]
\end{Thm}
These results can be found in  Han \cite[Theorem]{Han}, Jawerth \cite[Theorem 2.1]{Bj} and Sickel-Triebel \cite[Theorem 3.2.1]{SiTr} for the case $p=r.$
\subsection*{$\mathcal{L}$-superharmonic function}

We denote the positive part and negative one of a real-valued function $u$ by $u_{+}:=\max\{u,0\}$ and $u_{-}:=\max\{-u,0\},$ respectively.

\begin{Def}
A function $u:\R^{n}\to (-\infty,\infty]$ is said to be \textbf{$\mathcal{L}$-superharmonic} in $\Omega$ if it satisfies the following properties:
\begin{itemize}
\item[(i)] $u<+\infty$ almost everywhere in $\R^{n},$
\item[(ii)] $u$ is lower semicontinuous in $\Omega,$
\item[(iii)] for each $D\Subset\Omega$ and each weak solution $v \in \mathcal{C}(\bar{D})$ of $\mathcal{L}v=0$ in $D$ with $v_{+}\in L^{\infty}(\R^{n})$ such that $u\geq v$ on $\partial D,$ it holds that $u\geq v$ in $D,$
\item[(iv)] $u_{-}\in L^{p-1}_{\alpha p}(\R^{n}),$ where the tail space $L^{p-1}_{\alpha p}(\R^{n})$  is defined as 
\[
 L^{p-1}_{\alpha p}(\R^{n})=\Big\{u\in L^{p-1}_{\text{loc}}(\R^{n}):  \int_{\R^{n}} \frac{|u(y)|^{p-1}}{(1+|y|)^{n+\alpha p}} dy <+\infty    \Big\}.
 \]
\end{itemize}
\end{Def}
The notation $D\Subset \Omega$ means that $\bar{D}$ is a compact subset of $\Omega$.

 
\begin{Def}
A function $u\in W_{\text{loc}}^{\alpha,p}(\Omega)$ such that $u_{-}$ belongs $L^{p-1}_{\alpha p}(\R^{n})$ is a weak $\mathcal{L}$-supersolution of $\mathcal{L}u=0$
if
\[
<\mathcal{L}u,\phi>=\int_{\R^{n}}\int_{\R^{n}} |u(x)-u(y)|^{p-2}(u(x)-u(y)) (\phi(x)-\phi(y)) k(x,y) dxdy\geq 0
\]
for every nonnegative $\phi\in\mathcal{C}^{\infty}_{0}(\Omega).$
\end{Def}
\begin{Def}
A function $u$ is said to be $\mathcal{L}$-harmonic in $\Omega$ if it is a weak solution of $\mathcal{L}u=0$ in $\Omega$ that is continuous in $\Omega$.
\end{Def}

\begin{Thm}   [{\text See \cite{kimlee}}]    \label{thm:propsuper1} 
Let $u$ be a weak supersolution of $\mathcal{L}u=0$ in $\Omega$. Then
\[
u(x)= \mathrm{ess}\liminf_{{y\to x}} u(y)
\]
for a.e. $x\in\Omega$. In particular, $u$ has a representative that is lower semicontinuous in $\Omega$.
\end{Thm}

\begin{Thm}   [{\text See \cite{kimlee}}]    \label{thm:propsuper2} 
If $u$ is an $\mathcal{L}$-superharmonic function in $\Omega,$ then 
\[
u(x)= \liminf_{{{y\to x}} } u(y)= \mathrm{ess}\liminf_{{y\to x}} u(y)\quad\text{for every}\quad x\in\Omega.
\]
If, in addition, $u$ is locally bounded in $\Omega$ or $u\in W_{\mathrm{loc}}^{\alpha,p}(\Omega)$, then $u$ is a weak supersolution in $\Omega$.
\end{Thm}

\begin{Thm}   [{\text See \cite{kimlee}}]    \label{thm:propsuper3} 
If $u$ is a  weak supersolution of $\mathcal{L}u=0$ in $\Omega$ that is lower semicontinuous  satisfying
\[
u(x)= \mathrm{ess}\liminf_{{y\to x}} u(y)
\]
for every $x\in\Omega$, then $u$ is $\mathcal{L}$-superharmonic in $\Omega$.
\end{Thm}

\begin{Cor}
If $u$ is a supersolution of $\mathcal{L}u=0$ in $\Omega$, then there is a superharmonic function $v$ in $\Omega$ such that $u=v$ a.e.
\end{Cor}

\begin{Thm}   [{\text See \cite{MKS}}]    \label{thm:rieszmeasure} 
Let $u$ be superharmonic  in $\Omega$. Then there exists a unique nonnegative locally finite Borel measure $\mu$ in $\Omega$ such that
\[
\mathcal{L}u=\mu
\]
in the sense of distribution, that is,
\[
<\mathcal{L}u,\phi>=\int_{{\R}^{n}} \phi \ d\mu
\]
holds whenever $\phi\in\mathcal{C}^{\infty}_{0}(\Omega)$, if $u=0$ in $\R^{n}\setminus\Omega$.
\end{Thm}
The measure $\mu$ given in the above Theorem \ref{thm:rieszmeasure} is called \textbf{Riesz measure}.

\subsection*{Some recent results on potentials}
The proof of main Theorem \ref{mainthm} is based on pointwise estimates of $\mathcal{L}$-superharmonic functions in terms of Wolff potentials. We recall the definition of Wolff potential with fractional orders (see \cite{AH}, \cite{KMS}).

For $\alpha>0, p>1$ such that $\alpha p<n,$ the \textbf{Wolff potential} $\mathbf{W}_{\alpha,p}\sigma(x)$ is defined by
\[
 \mathbf{W}_{\alpha,p}\sigma(x)=\int_{0}^{\infty} \Big[\frac{\sigma(B(x,r))}{r^{n-\alpha p} }\Big]^{\frac{1}{p-1}}\frac{dr}{r},\quad x\in\R^{n}.
 \]
where $\sigma\in\mathcal{M}^{+}(\R^{n})$. In the case when $p=2$, it reduces (up to a normalization constant) to the Riesz potential of order $\alpha$ for $0<\alpha<n$, which can be defined as follows
 \[
 \mathbf{I}_{\alpha}\sigma(x)=C(n,\alpha)\int_{\R^{n}}\frac{\sigma(B(x,r))}{r^{n-\alpha}}\frac{dr}{r},\quad x\in \R^{n},
 \]
with a normalization constant
\[
C(n,\alpha)=\frac{\Gamma(\frac{n-\alpha}{2})}{2^{\alpha}\pi^{\frac{n}{2}}\Gamma\frac{\alpha}{2}}.
\] 
The Riesz potential satisfies the following semigroup property 
\[
\mathbf{I}_{\alpha}\mathbf{I}_{\beta}f=\mathbf{I}_{\alpha+\beta}f,\quad\text{for}\quad \alpha,\beta>0, \alpha+\beta<n.
\]
\begin{Rem}
The Riesz potential $\mathbf{I}_{\alpha}$ is an isomorphism from $F^{s}_{q}[L^{p,r}]$ onto $F^{s-\alpha}_{q}[L^{p,r}]$ and 
\be\label{riesz}
\|\mathbf{I}_{\alpha}f\|_{F^{s}_{q}[L^{p,r}]}=\|f\|_{F^{s-\alpha}_{q}[L^{p,r}]},
\ee
see \cite{BCT}.
\end{Rem}
In our case, we will use the following elementary inequality when we deduce from Wolff potential to Riesz potential: for every $R>0$,
\[
\Big( \int_{0}^{R}\big( \frac{\phi(r)}{r^{\gamma}}\big)^{\frac{1}{p-1}}\frac{dr}{r}\Big)^{p-1}\leq c(p,\gamma)\int_{0}^{2R}\frac{\phi(r)}{r^{\gamma}}\frac{dr}{r},
\]
where $\gamma>0, 1<p<2,$ and $\phi$ is a non-decreasing function on $(0,\infty)$. In particular, we have
	\bea\label{ineq}
\mathbf{W}_{\alpha,p}\sigma(x)&&=\int_{0}^{\infty} \Big[\frac{\sigma(B(x,r))}{r^{n-\alpha p} }\Big]^{\frac{1}{p-1}}\frac{dr}{r}\nonumber\\
&&\leq c \Big[ \int_{0}^{\infty}\frac{\sigma(B(x,r))}{r^{n-\alpha p} }\frac{dr}{r}\Big]^{\frac{1}{p-1}}=c[\mathbf{I}_{\alpha p}\sigma(x)]^{\frac{1}{p-1}}.
\eea
\begin{Thm}[{\text See \cite{MKS}}] \label{Pointwiseupperbound}
Let $n\in\N,\ 0<\alpha<1$ and $1<p\leq\frac{n}{\alpha}$. Let $\mu\in\mathcal{M}^{+}(\Omega)$ where $\Omega$ is an bounded open set in $\R^{n}$. Let $u$ is an $\mathcal{L}$-superharmonic function $\R^{n}$ with $u=0$ in $\R^{n}\setminus\Omega$. If $\mathcal{L}u=\mu,$ then there exists $C=C(n,p,q,\alpha)>0$ such that
\[\label{global}
C^{-1}\mathbf{W}_{\alpha,p}\mu(x)\leq u(x)\leq C \mathbf{W}_{\alpha,p}\mu(x)\quad\text{for all}\quad x\in\R^{n}.
\]
\end{Thm}

\section{Proof of Theorem \ref{mainthm}}\label{sec3}
\begin{proof}
Denote by $\mu=\mu[u]$  the Riesz measure associated with the $\mathcal{L}$-superharmonic function $u$. Employing Theorem \ref{Pointwiseupperbound} and the  inequality \eqref{ineq}, we find that
\begin{align*}
\|u\|_{L^{\frac{nq}{n-\alpha q},t}(\R^{n})}
&\leq c\|\W_{\alpha,p}\mu(x)\|_{L^{\frac{nq}{n-\alpha q},t}(\R^{n})}\\
&\leq\|[ \I_{\alpha p}\mu(x)]^{\frac{1}{p-1}}\|_{L^{\frac{nq}{n-\alpha q},t}(\R^{n})}\\
&=\|\I_{\alpha p}\mu(x)\|_{L^{\frac{nq}{(n-\alpha q)(p-1)},\frac{t}{p-1}}(\R^{n})}^{\frac{1}{p-1}}.
\end{align*}
By the seminorm property of Riesz potential, we have
\[
\|u\|_{L^{\frac{nq}{n-\alpha q},t}(\R^{n})}\leq \| \I_{\alpha (p-1)}(\I_{\alpha}\mu(x)) \|_{L^{\frac{nq}{(n-\alpha q)(p-1)},\frac{t}{p-1}}(\R^{n})}^{\frac{1}{p-1}}.
\]
The integral on the right-hand side of the preceding inequality can be written as the Triebel-Lizorkin-Lorentz norm \cite[Proposition 2.3.5] {Tr}, we get
\begin{align*}
\| \I_{\alpha (p-1)}(\I_{\alpha}\mu(x))\|_{L^{\frac{nq}{(n-\alpha q)(p-1)},\frac{t}{p-1}}(\R^{n})}
&=\| \I_{\alpha (p-1)}(\I_{\alpha}\mu(x))\|_{F^{0}_{2}[L^{\frac{nq}{(n-\alpha q)(p-1)},\frac{t}{p-1}}](\R^{n})}\\
&=\| \I_{\alpha}\mu(x)\|_{F^{-\alpha(p-1)}_{2}[L^{\frac{nq}{(n-\alpha q)(p-1)},\frac{t}{p-1}}](\R^{n})}.
\end{align*}
Then, we apply the embedding Theorem \ref{Lorentzembedding}, the last display implies that
\[
\| \I_{\alpha}\mu(x)\|_{F^{-\alpha(p-1)}_{2}[L^{\frac{nq}{(n-\alpha q)(p-1)},\frac{t}{p-1}}](\R^{n})}\leq \|\I_{\alpha}\mu(x)\|_{F^{0}_{\frac{q}{p-1}}[L^{\frac{q}{p-1},\frac{t}{p-1}}](\R^{n})}.
\]
We have a fact that the seminorm in Triebel-Lizorkin-Lorentz spaces can be characterized by a supremum over test functions $\phi \in F^{0}_{(\frac{q}{p-1})'}[L^{(\frac{q}{p-1})',(\frac{t}{p-1})'}]$ where $(\frac{q}{p-1})'=\frac{q}{q-p+1}$ and $(\frac{t}{p-1})'=\frac{t}{t-p+1},$ we get
\begin{align*}
\|\I_{\alpha}\mu(x)\|_{F^{0}_{\frac{q}{p-1}}[L^{\frac{q}{p-1},\frac{t}{p-1}}](\R^{n})}&=\sup_{[\phi]_{F^{0}_{(\frac{q}{p-1})'}[L^{(\frac{q}{p-1})',(\frac{t}{p-1})'}]}\leq 1}\Big|\int_{\R^{n}} \I_{\alpha}\mu(x)\ \phi dx\Big|\\
&\leq\sup_{[\phi]_{F^{0}_{(\frac{q}{p-1})'}[L^{(\frac{q}{p-1})',(\frac{t}{p-1})'}]}\leq 1}\int_{\R^{n}} | (\I_{\alpha}\phi)(x)|\ d\mu\\
&=\sup_{[\phi]_{F^{0}_{(\frac{q}{p-1})'}[L^{(\frac{q}{p-1})',(\frac{t}{p-1})'}]}\leq 1}<\mathcal{L}u, \I_{\alpha}\phi>\\
&\leq \Lambda\int_{\R^{n}}\int_{\R^{n}} \frac{|u(x)-u(y)|^{p-1}  |\I_{\alpha}\phi(x)-\I_{\alpha}\phi(y)|}{|x-y|^{n+\alpha p}}dxdy\\
&= \Lambda\int_{\R^{n}}\int_{\R^{n}} \frac{|u(x)-u(y)|^{p-1}}{|x-y|^{(\frac{n}{q}+\alpha)(p-1)}}\frac{|\I_{\alpha}\phi(x)-\I_{\alpha}\phi(y)|}{|x-y|^{\frac{n(q-p+1)}{q}+\alpha}}dxdy.
\end{align*}
Applying the H\"{o}lder's inequality in the Lorentz space with the exponents $\frac{q}{p-1}, \frac{t}{p-1}, (\frac{q}{p-1})'$ and $(\frac{t}{p-1})'$ respectively, we  obtain
\begin{align*}
\|\I_{\alpha}\mu(x)\|_{F^{0}_{\frac{q}{p-1}}[L^{\frac{q}{p-1},\frac{t}{p-1}}](\R^{n})}&\leq \Big\|  \frac{|u(x)-u(y)|^{p-1}}{|x-y|^{(\frac{n}{q}+\alpha)(p-1) }} \Big\|_{L^{\frac{q}{p-1},\frac{t}{p-1}}(\R^{n}\times\R^{n})}\\
&\quad\;\Big\| \frac{|\I_{\alpha}\phi(x)-\I_{\alpha}\phi(y)|}{|x-y|^{\frac{n(q-p+1)}{q}+\alpha }} \Big\|_{L^{(\frac{q}{p-1})',(\frac{t}{p-1})'}(\R^{n}\times\R^{n})}.
\end{align*}
Then, we find that
\begin{align*}
\|\I_{\alpha}\mu(x)\|_{F^{0}_{\frac{q}{p-1}}[L^{\frac{q}{p-1},\frac{t}{p-1}}](\R^{n})}&=\Big\|  \frac{|u(x)-u(y)|}{|x-y|^{\frac{n}{q}+\alpha }} \Big\|_{L^{q,t}(\R^{n}\times\R^{n})}^{p-1}\\
&\quad\;\Big\| \frac{|\I_{\alpha}\phi(x)-\I_{\alpha}\phi(y)|}{|x-y|^{\frac{n(q-p+1)}{q}+\alpha }} \Big\|_{L^{(\frac{q}{p-1})',(\frac{t}{p-1})'}(\R^{n}\times\R^{n})}\\
&=[u]_{W^{\alpha,(q,t)}(\R^{n})}^{p-1} [\I_{\alpha}\phi]_{W^{\alpha, ((\frac{q}{p-1})',(\frac{t}{p-1})') }(\R^{n})}.
\end{align*}
Hence, we conclude the desired estimate
\[
\|u\|_{L^{\frac{nq}{n-\alpha q},t}(\R^{n})} \leq C(n,\alpha,p,q,t,\Lambda) [u]_{W^{\alpha,(q,t)}(\R^{n})}.
\]
\end{proof}


\section{Further Results}\label{sec4}
In this section, we present additional results derived from our methods.
Observe that $L^{\frac{nq}{n-\alpha q}}(\R^{n})\subset L^{\frac{nq}{n-\alpha q},t}(\R^{n})$ when $t>\frac{nq}{n-\alpha q}$. As an immediate consequence of \eqref{fracSoboineq}, we get the following theorem. 
\begin{Thm}\label{mainthm2}
Let $\Omega$ be an open bounded subset in $\R^{n}$ where $n\geq 1$. Let $0<\alpha<1, 1<p<2$ and $p-1<q<1.$ For any $\mathcal{L}$-superharmonic function $u$ on $\Omega$ with $u=0$ in $\R^{n}\setminus\Omega$, we have
\[
\|u\|_{L^{\frac{nq}{n-\alpha q}}(\R^{n})} \leq C [u]_{W^{\alpha,q}(\R^{n})}
\]
where $C=C(n,\alpha,p,q,\Lambda)>0.$
\end{Thm}
This inequality follows directly from the inequality \eqref{fracSoboineq} by setting $\frac{nq}{n-\alpha q}<t\leq \infty$. Since the proof structure is the same as the proof of Theorem \ref{mainthm}, we omit a full proof of Theorem \ref{mainthm2}.

\begin{Rem}
We observe that the sublinear fractional Sobolev inequality established from Theorem \ref{mainthm2} corresponds to the classical result of \cite{Ph1} by Phuc, in light of the following result:
\[
\liminf\limits_{\alpha\to 1^{-}}(1-\alpha)\int_{\R^{n}}\int_{\R^{n}}\frac{|u(x)-u(y)|^{q}}{|x-y|^{n+\alpha q}}dxdy=C(n,p)\int_{\R^{n}}|\nabla u|^{q}dx 
\]
as shown in \cite[Theorem 7.39]{Le}.
\end{Rem}

The classical Gagliardo-Nirenberg interpolation has been widely extended to the fractional setting, playing a fundamental role in the study of intermediate regularity between Sobolev-type spaces. In particular, a refined version involving the borderline Sobolev space  $W^{1,1}(\R^{n})$ was established by \cite{Cohen03}, where the inequality
\be\label{Gag-Niren}
[u]_{W^{\alpha_{1},{q_{1}}}(\R^{n})} \leq [u]_{W^{1,{1}}(\R^{n})}^{1-\theta} [u]_{W^{\alpha_{3},{q_{3}}}(\R^{n})}^{\theta},\quad\forall u\in\mathcal{C}_{0}^{\infty}(\R^{n})
\ee
holds for every $\theta\in(0,1),$ provided that $ \alpha_{3}\in(0,1),$ $q_{3}\in(1,\infty)$ and the condition $\alpha_{3}q_{3}<1$ is satisfied. The interpolation indices are given by the convexity relations:
\[
\alpha_{1}=\theta\alpha_{3}+(1-\theta)\quad\text{and}\quad\frac{1}{q_{1}}=\frac{\theta}{q_{3}}+(1-\theta).
\]
However, as shown in \cite{BM2018}, such interpolation \eqref{Gag-Niren} fails when $\alpha_{3}q_{3}\geq 1.$ Recent studies \cite{BSY} have explored such interpolation inequalities in different spaces like Lorentz spaces, weak $L^{p}$-spaces or weighted Lebesgue spaces. In what follows, we turn our attention to the Gagliardo-Nirenberg interpolation that may hold for fractional derivatives in Lorentz spaces. 
\begin{Thm}\label{interpolation}
Let $\alpha_{i}\in(0,1)$, $1<p<2,$ $q_{i}\in(p-1,1)$ and $t_{i}\in(p-1, \infty]$ for $i=1,2,3.$ Assume $0<\theta<1$ satisfies the following conditions
\[
\alpha_{1}=(1-\theta)\alpha_{2}+\theta\alpha_{3},\,\,\frac{1}{q_{1}}=\frac{1-\theta}{q_{2}}+\frac{\theta}{q_{3}}\quad\text{and}\quad \frac{1}{t_{1}}=\frac{1-\theta}{t_{2}}+\frac{\theta}{t_{3}}.
\]
Then the interpolation inequality
\[
\|u\|_{L^{\frac{nq_{1}}{n-\alpha_{1} q_{1}},t_{1}}(\R^{n})}\leq [u]_{W^{\alpha_{2},{(q_{2},t_{2})}}(\R^{n})}^{1-\theta} [u]_{W^{\alpha_{3},{(q_{3},t_{3})}}(\R^{n})}^{\theta}
\]
with holds for all $\mathcal{L}$-superharmonic functions $u\in W^{\alpha_{2},(q_{2},t_{2})}(\R^{n})\cap W^{\alpha_{3},(q_{3},t_{3})}(\R^{n})$ with $u=0$ in $\R^{n}\setminus\Omega.$
\end{Thm}
\begin{proof}
Assume that $u$ is $\mathcal{L}$-superharmonic functions in $W^{\alpha_{2},(q_{2},t_{2})}(\R^{n})\cap W^{\alpha_{3},(q_{3},t_{3})}(\R^{n})$. Then by Theorem \ref{mainthm} , we have 
\[
\|u\|_{L^{\frac{nq_{1}}{n-\alpha_{1} q_{1}},t_{1}}(\R^{n})}\leq \Big\|  \frac{|u(x)-u(y)|}{|x-y|^{\frac{n}{q}+\alpha }} \Big\|_{L^{q_{1},t_{1}}(\R^{n}\times\R^{n})}.
\]
Choosing the H\"{o}lder exponents $A=\frac{q_{1}q_{3}}{q_{3}-\theta q_{1}}, A'=\frac{q_{3}}{\theta}, B=\frac{t_{1}t_{3}}{t_{3}-\theta t_{1}}, $ and $ B'=\frac{t_{3}}{\theta}$ such that
\[
\frac{1}{q_{1}}=\frac{1}{A}+\frac{1}{A'}\quad \text{and}\quad \frac{1}{t_{1}}=\frac{1}{B}+\frac{1}{B'}
\]
where $0<q_{i},t_{i}\leq \infty$ for $i=1,2,3$, we estimate the left-hand side of the above display. Hence, we find that
\begin{align*}
\Big\|  \frac{|u(x)-u(y)|}{|x-y|^{\frac{n}{q}+\alpha }} \Big\|_{L^{q_{1},t_{1}}(\R^{n}\times\R^{n})}&=\Big\|  \frac{|u(x)-u(y)|^{1-\theta}}{|x-y|^{\frac{1-\theta}{q_{2}}n+(1-\theta)\alpha_{2} }} \frac{|u(x)-u(y)|^{\theta}}{|x-y|^{\frac{\theta}{q_{3}}n +\theta\alpha_{3}}}\Big\|_{L^{q_{1}, t_{1}}(\R^{n}\times\R^{n})}\\
&\leq \Big\| \frac{|u(x)-u(y)|^{1-\theta}}{|x-y|^{\frac{1-\theta}{q_{2}}n+(1-\theta)\alpha_{2} }}\Big\|_{L^{\frac{q_{1}q_{3}}{q_{3}-\theta q_{1}},\frac{t_{1}t_{3}}{t_{3}-\theta t_{1}}}(\R^{n}\times\R^{n})}\\
&\qquad\Big\|\frac{|u(x)-u(y)|^{\theta}}{|x-y|^{\frac{\theta}{q_{3}}n +\theta\alpha_{3}}}\Big\|_{L^{\frac{q_{3}}{\theta},\frac{t_{3}}{\theta }}(\R^{n}\times\R^{n})}\\
&= \Big\| \frac{|u(x)-u(y)|}{|x-y|^{\frac{n}{q_{2}}+\alpha_{2} }}\Big\|^{1-\theta}_{L^{q_{2},t_{2}}(\R^{n}\times\R^{n})} \\
&\qquad\Big\| \frac{|u(x)-u(y)|}{|x-y|^{\frac{n}{q_{3}}+\alpha_{3} }}\Big\|^{\theta}_{L^{q_{3},t_{3}}(\R^{n}\times\R^{n})}
\end{align*}
where $q_{2}=\frac{q_{1}q_{3}(1-\theta)}{q_{3}-\theta q_{1}}$, $t_{2}=\frac{t_{1}t_{3}(1-\theta)}{t_{3}-\theta t_{1}}$, $q_{3}=\frac{q_{1}q_{2}\theta}{q_{2}-(1-\theta)q_{1}}$ and $t_{3}=\frac{t_{1}t_{2}\theta}{t_{2}-(1-\theta)t_{1}}$.
This completes the proof of Theorem \ref{interpolation}.
\end{proof}

Following the arguments in the proof of Theorem \ref{mainthm}, one can extend Theorem \ref{mainthm}  to the weighted version of the inequality.
We collect some definitions and embedding results on Muckenhoupt classes $\mathcal{A}_{p}$.

\begin{Def}
Let $w$ be a positive, locally integrable function on $\R^{n}$, and $1<p<\infty$. Then $w$ belongs to the Muckenhoupt class $\mathcal{A}_{p}$, if there exists a constant $0<A<\infty$ such that for all balls $B$ the following inequality holds:
\[
\Big( \frac{1}{|B|} \int_{B} w(x)dx\Big)^{\frac{1}{p}}\cdot \Big(\frac{1}{|B|}\int_{B}w(x)^{-\frac{p'}{p}}dx \Big)^{\frac{1}{p'}}\leq A,
\]
where $|B|$ stands for the Lebesgue measure of the ball $B$.
\end{Def}
The Muckenhoupt class $\mathcal{A}_{\infty}$ is given by
\[
\mathcal{A}_{\infty}=\bigcup_{p>1}\mathcal{A}_{p}.
\]
\begin{Thm} [{\text See \cite{Bui}}] \label{weightedTriebel}
Let $0<p<\infty,$ $\alpha\in\mathbb{R}$ and $w\in\mathcal{A}_{\infty}.$
\begin{enumerate}[label=(\arabic*)]
\item[(1)]If $0<q_{0}\leq q_{1}\leq \infty,$ then
\[
{F^{\alpha}_{p,q_{1}}(\R^{n},dw)}\hookrightarrow{F^{\alpha}_{p,q_{2}}(\R^{n},dw)}. 
\]
\item[(2)]  Assume that there are numbers $c>0$ and $d>0$, such that for all balls,
\[
w(B(x,r))\geq cr^{d},\quad 0<r\leq 1,\quad x\in\R^{n}.
\]
Let $0<p_{1}<p_{2}<\infty, -\infty<\alpha_{2}<\alpha_{1}<\infty$, with
\[
\alpha_{1}-\frac{d}{p_{1}}=\alpha_{2}-\frac{d}{p_{2}}.
\] Then
\[
{F^{\alpha_{1}}_{p_{1},\infty}(\R^{n},dw)}\hookrightarrow{F^{\alpha_{2}}_{p_{2},q}(\R^{n},dw)}. 
\]
\end{enumerate}
	\end{Thm}
	Here, $d$ is the scaling exponent of the weight function $w$. 
	\begin{Thm}\label{weighted}
Let $\Omega$ be an open bounded subset in $\R^{n}$ where $n\geq 1$. Let $0<\alpha<1, 1<p<2$ and $p-1<q<1.$ Let $w\in \mathcal{A}_{\infty}$ and there are numbers $c,d>0$ such that
\[
w(B(x,r))\geq cr^{d},\quad 0<r\leq 1,\quad x\in\R^{n}.
\]
For any $\mathcal{L}$-superharmonic function $u$ on $\Omega$ with $u=0$ in $\R^{n}\setminus\Omega$, we have
\[
\big( \int_{\R^{n}} u^{\frac{dq}{d-\alpha q}} w(x)dx\big)^{\frac{d-\alpha q}{dq}}\leq C \Big( \int_{\R^{n}}\int_{\R^{n}}\frac{|u(x)-u(y)|^{q}}{|x-y|^{d+\alpha q}}w(y)w(x) dydx\Big)^{\frac{1}{q}}
\]
where $C=C(d,\alpha,p,q,\Lambda)>0.$
\end{Thm}
\begin{proof}
The theorem can be proved using the arguments in the proof of Theorem \ref{mainthm} together with Theorem \ref{weightedTriebel}.
\end{proof}
\begin{Rem}
The above inequality is a natural extension from the unweighted case $w=1$, whereas $d=n$.
\end{Rem}

It is well known that the classical Poincar\'{e} inequality holds for a bounded Lipschitz domain $\Omega$. Specifically, if $1\leq q<n,$ then there exists a positive constant $C(\Omega)$ such that inequality 
\be\label{standardpoincare}
\inf_{a\in\mathbb{R}}\Big( \int_{\Omega}|x-a|^{\frac{nq}{n-q}}dx\Big)^{\frac{n-q}{q}} \leq C(\Omega)\Big(\int_{\Omega}|\nabla u|^{q}dx\Big)^{\frac{1}{q}}
\ee
holds for all $u\in W^{1,q}(\Omega).$ However, Buckley and Koskela \cite{Sp} provided a counterexample that this Poincar\'{e} inequality is not valid in the sublinear case $0<q<1,$ even if $\frac{nq}{n-q}$ is small, $\Omega$ is a ball, and $u$ is a smooth function. Nevertheless, Chua and Phuc \cite{CP} could prove this inequality holds in the sublinear case when restricted to the class of superharmonic functions. In particular, the main result of \cite{CP}, states that for $\Omega$ being a bounded or unbounded John domain, the inequality 
\[
\inf_{a\in\mathbb{R}}\Big( \int_{\Omega}|x-a|^{\frac{nq}{n-q}}wdx\Big)^{\frac{n-q}{q}} \leq C(\Omega)\Big(\int_{\Omega}|\nabla u|^{q}vdx\Big)^{\frac{1}{q}}
\]
holds for  the class of $\mathcal{A}$-superharmonic functions, where $w$ and $v$ are suitable weight functions.

In the context of the nonlocal level, the following fractional Sobolev-Poincar\'{e} inequality has been established, given by
\be\label{poincare}
\inf_{a\in\mathbb{R}}\Big( \int_{\Omega}|x-a|^{\frac{nq}{n-\alpha q}}dx\Big)^{\frac{n-\alpha q}{nq}} \leq C\Big(\int_{\Omega}\int_{\Omega} {\frac{{|u(x)-u(y)|}^{q}}{{|x-y|}^{n+\alpha q}}}dydx\Big)^{\frac{1}{q}}
\ee
for $u\in L^{q}(\Omega)$ under the conditions $0<\alpha<1, 2\leq 1<\frac{n}{\alpha}$, assuming that $\Omega\subset\R^{n}$ satisfying the measure density condition. On the other hand, if  $\Omega$ is John domain, then there exists a positive constant $C(n,\alpha,\delta,c,q)$ such that
 \[
\inf_{a\in\mathbb{R}}\Big( \int_{\Omega}|x-a|^{\frac{nq}{n-\alpha q}}dx\Big)^{\frac{n-\alpha q}{nq}} \leq C\Big(\int_{\Omega}\int_{\Omega\cap B(x, \delta d(x,\partial \Omega))} {\frac{{|u(x)-u(y)|}^{q}}{{|x-y|}^{n+\alpha q}}}dydx\Big)^{\frac{1}{q}}
\]
holds for every $u\in L^{1}(\Omega).$ However, this improved version of the fractional Sobolev-Poincar\'{e} inequality fails to hold for all bounded domains that satisfy the measure density. A counterexample has been found in \cite{BLV}. 
\begin{Rem}
It is worth noting that  the Hessian Sobolev inequality 
\[ \label{hessianineq}
\|u\|_{L^{\frac{n(k+1)}{n-2k}}(\R^{n})}\leq C \Big( \int_{(\R^{n})}|u| F_{k}[u] dx\Big)^{\frac{1}{k+1}}
 \]
can be derived by following the argument from Theorem \ref{mainthm}, with the choices $\alpha=\frac{2k}{k+1}$ and $p=k+1,$ where $1\leq k<\frac{n}{2}.$ Here,
 $F_{k}[u]$ is the Hessian operator with $k=1,2,\dots,n$ defined by
\[
F_{k}[u]=S_{k}(\lambda(D^{2}u)),         
\]
where $\lambda(D^{2}u)=(\lambda_{1},\cdots,\lambda_{n})$ are the eigenvalues of the Hessian matrix of second-order partial derivatives $D^{2}u,$ and $S_{k}$ is the $k^{th}$ symmetric function on $\R^{n}$
\[
S_{k}(\lambda)= \sum_{1\leq i_{1}<\cdots<i_{k}\leq n}\lambda_{i_{1}}\cdots\lambda_{i_{k}}.
\]

The reader can see the corresponding result and its derivation using duality and potential theory in \cite{V15}.
\end{Rem}

\medskip
\medskip
\subsection*{Open Problem}

From the authors' knowledge, the validity of the fractional Poincar\'{e} inequality for superharmonic functions in the sublinear case remains an open problem. While the classical Poincar\'{e} inequality has been established for superharmonic functions in the local setting (see \cite{CP}), its nonlocal counterpart, particularly in the framework of fractional Sobolev spaces, has not yet been fully explored. Specifically, given a domain $\Omega\subset\R^{n},$ we ask whether there exists a constant $C$ such that the fractional Sobolev-Poincar\'{e} inequality 
\[
\inf_{a\in\mathbb{R}}\Bigg( \int_{\Omega}|x-a|^{\frac{nq}{n-\alpha q}}dx\Bigg)^{\frac{n-\alpha q}{nq}} \leq C\Bigg(\int_{\Omega}\int_{\Omega} \frac{|u(x)-u(y)|^{q}}{|x-y|^{n+\alpha q}} dy dx\Bigg)^{\frac{1}{q}}
\]
holds for the class of $\mathcal{L}$-superharmonic functions $u$, particularly when $p-1<q<1$ and $0<\alpha<1.$




\subsection*{Data Availability Statement}
The authors do not analyze or generate any datasets because this work proceeds with a theoretical and mathematical approach. The relevant materials can be obtained from the references below.

\subsection*{Conflict of interest} 
On behalf of all authors, the corresponding author states that there is no conflict of interest.

\subsection*{Acknowledgement}
This study was supported by Thammasat University Research Fund, Contract No. TUFT 072/2568.  
A.C.M. acknowledges financial support from the Excellent Foreign Student (EFS) scholarship,  Sirindhorn International Institute of Technology (SIIT),  Thammasat University.
The authors thank Armin Schikorra and Nguyen Cong Phuc for their valuable discussions and comments.
\bibliographystyle{abbrv} 
\bibliography{reference(MS1)}
\end{document}